\documentclass{amsart}
\usepackage{amsmath}
\usepackage{amssymb}
\usepackage{amsthm}
\usepackage[msc-links,abbrev]{amsrefs}
\usepackage[noabbrev,capitalize]{cleveref}

\DeclareMathOperator{\tr}{tr}

\def\sideremark#1{\ifvmode\leavevmode\fi\vadjust{\vbox to0pt{\vss
 \hbox to 0pt{\hskip\hsize\hskip1em
 \vbox{\hsize3cm\tiny\raggedright\pretolerance10000
 \noindent #1\hfill}\hss}\vbox to8pt{\vfil}\vss}}}

\newtheorem{theorem}{Theorem}[section]
\newtheorem{proposition}[theorem]{Proposition}
\newtheorem{lemma}[theorem]{Lemma}
\newtheorem{corollary}[theorem]{Corollary}

 \theoremstyle{definition}
 \newtheorem{definition}[theorem]{Definition}

 \theoremstyle{remark}
 \newtheorem{remark}[theorem]{Remark}

\numberwithin{equation}{section}

\makeatletter
\renewcommand\subsubsection{\@secnumfont}{\bfseries}%
\renewcommand\subsubsection{\@startsection{subsubsection}{3}
\z@{.5\linespacing\@plus.7\linespacing}{-.5em}%
{\normalfont\bfseries}}
\makeatother
\RequirePackage[margin=1.3in]{geometry}
\usepackage{comment}
\usepackage{xcolor}
\usepackage{tensor}
\usepackage[rightcaption]{sidecap}
\usepackage{graphicx} 
\graphicspath{ {images/} }
\usepackage{wrapfig}
\usepackage{float}
\RequirePackage{tcolorbox}

\usepackage{tikz}
\usetikzlibrary{arrows.meta, calc}

\newcommand{\Ric}{\text{Ric}}

\newcommand{\Div}{\text{div}}
\newcommand{\Vol}{\text{Vol}}
\newcommand{\Hyp}{\mathbb{H}^{n+1}(-1)}

\newcommand{\A}{\alpha}
\newcommand{\B}{\beta}
\newcommand{\G}{\gamma}
\newcommand{\D}{\delta}
\newcommand{\E}{\epsilon}

\begin{document}

\title[Eigenvalue Estimates on AH Manifolds]{Sharp Spectral Bounds for the $p$-Laplacian and Polyharmonic Operators on Asymptotically Hyperbolic Manifolds}
\author{Samuel P\'erez-Ayala}
\address{Samuel P\'erez-Ayala\\370 Lancaster Ave \\Haverford College  \\ Haverford\\PA 19041\\ USA}
\email{sperezayal@haverford.edu}

\subjclass[2020]{}
\begin{abstract}
We derive sharp bounds for three types of eigenvalue problems. First, we derive an upper bound for the first $p$-Dirichlet eigenvalue on conformally compact (CC) spaces. As a consequence, we show that for a class of CC submanifolds of asymptotically hyperbolic spaces, the asymptotic sectional curvatures, the meeting angle at infinity, and the vanishing of the norm of the mean curvature are all determined by its first $p$-Dirichlet eigenvalue. Additionally, we derive sharp upper bounds for the first eigenvalue of polyharmonic operators under both clamped and buckling boundary conditions. Finally, we prove sharp lower bounds for all three types of eigenvalue problems on weakly Poincar\'e-Einstein spaces with $\Ric_{g_+}\ge -ng_+$ and whose conformal infinity has nonnegative Yamabe constant, and on their submanifolds. 
\end{abstract}

\maketitle

\setcounter{tocdepth}{3}
\makeatletter
\renewcommand{\l@section}{\@tocline{1}{0pt}{0pt}{}{}}
\renewcommand{\l@subsection}{\@tocline{2}{0pt}{1.5em}{}{}}
\renewcommand{\l@subsubsection}{\@tocline{3}{0pt}{3em}{}{}}
\makeatother
\tableofcontents
\section{Introduction}
\subsection{Three eigenvalue problems} We are concerned with three types of eigenvalue problems: the eigenvalue problem for the $p$-Laplacian operator ($1<p<\infty$) under Dirichlet boundary conditions, the eigenvalue problem for the biharmonic operator under buckling boundary conditions, and the eigenvalue problem for polyharmonic operators under clamped boundary conditions. Most of these problems were initially introduced for Euclidean domains, and their relevance comes from multiple areas in science; see \cite{Ashbaugh1999}, and references therein, for an introduction to these problems and motivation. Here, we are concerned with the bottom of the spectrum of these operators on a large class of complete and noncompact manifolds $(X,g)$. We proceed with some definitions.

Let $\Omega\subset X$ be a bounded domain with a smooth boundary $\partial \Omega$. The Dirichlet eigenvalue problem for the $p$-Laplacian operator on $\Omega$ is finding a nontrivial solution of
\begin{equation}\label{p-Equation}
    \begin{cases}
        \Delta_p f + \lambda |f|^{p-2}f = 0 & \quad\text{ in }\Omega,\\
        f = 0 & \quad\text{ on }\partial\Omega,
    \end{cases}
\end{equation}
where $\Delta_p f = \text{div}(|\nabla_{g}f|^{p-2}\nabla_{g}f)$ is the so-called $p$-Laplace operator and where $1<p<\infty$. Note that $\Delta_2=\Delta_{g}$ is linear as it is just the standard Laplace-Beltrami operator. Our focus is on the first $p$-Dirichlet eigenvalue, whose variational characterization is given by
\begin{equation}\label{p-Eigenvalue}
    \lambda_{1,p}(\Omega) = \inf_{f\in W^{1,p}_0(\Omega)\setminus\{0\}} \frac{\displaystyle\int_\Omega |\nabla_{g}f|^p\;dv_{g}}{\displaystyle\int_\Omega |f|^p\;dv_{g}}
\end{equation}
where, and in more generality, $W^{k,p}_0(\Omega)$ is the completion of $C^\infty_0(\Omega)$ under the Sobolev norm $\|\cdot\|_{L^{k,p}(g)}$. The reader can consult \cites{LEAn2006Epft,LindqvistPeter}, and references therein, for the definition of the first Laplace $p$-Dirichlet eigenvalue (\ref{p-Eigenvalue}) and basic properties of the associated eigenfunction. 

Second, we study the following eigenvalue problem for polyharmonic operators: 
\begin{equation}\label{ClampedProblem}
\begin{cases}
    (-\Delta_{g})^lf = \Gamma f & \text{ in }\Omega \\ f = \partial_\nu f =\cdots = \partial^{l-1}_\nu f =0 &\text{on }\partial \Omega,
\end{cases}
\end{equation}
where $l\in\mathbb{N}$ and $\nu$ is the inner unit normal vector field along $\partial \Omega$. When $l = 1$, (\ref{ClampedProblem}) corresponds to the standard Dirichlet eigenvalue problem on $\Omega$ for the Laplace operator, whereas if $l=2$ (\ref{ClampedProblem}), it is related to the so called oscillating clamped plate \cite{JostJurgen2011Ubfe}. The operator $(-\Delta_{g})^2$ is called the biharmonic operator and, more generally, we refer to $(-\Delta_{g})^l$ as the polyharmonic operator of order $2l$. In conformal geometry, polyharmonic operators are relevant because they constitute the principal part of the so called GJMS operators \cite{GJMS}.

The boundary conditions (\ref{ClampedProblem}) are often referred to as Dirichlet-like or Dirichlet boundary conditions or clamped plate boundary conditions. We will adopt the latter term. The spectrum of $(-\Delta_{g})^l$ is discrete, real, and positive, and its first eigenvalue $\lambda_1^C((-\Delta_{g})^l,\Omega)$ can be characterized variationally as
\begin{equation}\label{VariationalCharacterization}
    \lambda_1^C((-\Delta_{g})^l,\Omega) = \inf_{f\in W^{2,l}_0(\Omega)\setminus\{0\}}R^C_l(f),
\end{equation}
where 
\begin{equation}\label{Rayleigh-Quotient-C}
    R^C_l(f) = 
\begin{cases}
    \frac{\displaystyle\int_\Omega |\Delta_{g}^mf|^2\;dv_{g}}{\displaystyle\int_\Omega |f|^2\;dv_{g}} & \text{if } l=2m, \\ \frac{\displaystyle\int_\Omega |\nabla_{g}\Delta_{g}^mf|^2\;dv_{g}}{\displaystyle\int_\Omega |f|^2\;dv_{g}} & \text{if }l=2m+1;
\end{cases}
\end{equation}
see for instance \cites{ZhangLiuwei2016Tlbo, BuosoLamberti2013}.

The third eigenvalue problem we consider is the so called buckling eigenvalue problem:
\begin{equation}\label{BucklingProblem}
\begin{cases}
    (-\Delta_{g})^2f = \Lambda \Delta_{g}f & \text{ in }\Omega \\ f = \partial_\nu f =0 &\text{on }\partial \Omega.
\end{cases}
\end{equation}
Its first eigenvalue $\lambda_1^B((-\Delta_{g})^2,\Omega)$ is characterized variationally as 
\begin{equation}
    \lambda_1^B((-\Delta_{g})^2,\Omega) = \inf_{f\in W^{2,2}_0(\Omega)\setminus\{0\}} R^B(f),
\end{equation}
where 
\begin{equation}\label{Rayleigh-Quotient-B}
    R^B(f) = \frac{\displaystyle\int_\Omega |\Delta_{g}f|^2\;dv_{g}}{\displaystyle\int_\Omega |\nabla_{g}f|^2\;dv_{g}}.
\end{equation}
This is discussed, for instance, in \cite{BuosoLamberti2013}.

Note that, in all the cases, these eigenvalue problems are monotonic nonincreasing with respect to domain inclusion. For a noncompact manifold $(X,g)$, the first $p$-Dirichlet eigenvalue is defined as \[\lambda_{1,p}(X,g):=\inf_{\Omega\subset X}\lambda_{1,p}(-\Delta_{g}, \Omega),\] where $\Omega$ ranges over all the bounded domains with smooth boundary; see \cite{PerezAyalaTyrrell} and references therein. For $p=2$, this is just the Laplacian operator, which is the polyharmonic operator with $l=1$. For the biharmonic operator ($l=2$), the first Dirichlet eigenvalue of a noncompact manifold $(X,g)$ is defined as \[\inf_{\Omega\subset X}\lambda_1^C((-\Delta_{g})^2, \Omega);\] see \cites{LinHezi2025Sgef,KristalyAlexandru2020Ftoc, FarkasKajantoKristaly2025} and references therein. More generally,

\begin{definition}
We define the first clamped eigenvalue of the polyharmonic operator $(-\Delta_g)^l$ on $(X,g)$ as 
\begin{equation}\label{Poly-Eigenvalue-Def}
    \Gamma^l_1(X,g):= \inf_{\Omega\subset X}\lambda_1^C((-\Delta_{g})^l, \Omega),
\end{equation}
where $\Omega$ ranges over all bounded domain with smooth boundary. Similarly, we define the first buckling eigenvalue on $(X^{n+1},g)$ as 
\begin{equation}
    \Lambda_1(X,g):=\inf_{\Omega\subset X} \lambda_1^B((-\Delta_{g})^2,\Omega),
\end{equation}
where $\Omega$ ranges over all the bounded domains with a smooth boundary.
\end{definition}

In this work, we are concerned with first eigenvalue estimates on conformally compact manifolds, asymptotically hyperbolic manifolds, and their submanifolds. We proceed to explain what these are and some of their key properties. Additionally, we provide some examples. 

\subsection{Conformally compact manifolds - definitions and examples}

The model case for an asymptotically hyperbolic manifold, an important subclass of conformally compact manifolds, is hyperbolic space $\mathbb{H}^{n+1}(-1)$, where we use the Poincar\'e ball model. That is, if we denote by $\mathbb{B}^{n+1}$ the open unit ball in $\mathbb{R}^{n+1}$, then $\mathbb{H}^{n+1}(-1) = (\mathbb{B}^{n+1}, 4(1-|x|^2)^{-2}g_{\mathbb{R}^{n+1}}) = (\mathbb{B}^{n+1}, g_H)$. The asymptotically hyperbolic manifolds are those that resemble $\mathbb{H}^{n+1}(-1)$ outside some compact set. In what follows, we formalize this intuition, discuss key aspects and provide some examples. 

\begin{definition}[Conformally Compact Manifold]
Let $X^{n+1}$ be the interior of a smooth and compact manifold $\overline X$ with smooth and nonempty boundary $\partial X$, and let $g_+$ be a complete metric on $X$. We call a defining function a smooth function $r:\overline X\to [0,\infty)$ with $r^{-1}({0}) = \partial X$ and $dr\not = 0$ on $\partial X$. The manifold $(X^{n+1},g_+)$ is said to be conformally compact (CC) of order $C^{m,\alpha}$ if there exists a defining function $r$ for its boundary $\partial X$ such that $\overline g = r^2g_+$ extends to a $C^{m,\alpha}$ metric on $\overline{X}$. Throughout this work, we always assume that our CC manifolds are $C^{3,\alpha}$, $\alpha\in(0,1)$.
\end{definition}
If $\rho$ is any other defining function, then the induced metric $(\rho^2g_+)|_{(T\partial X)^2}$ on the boundary $\partial X$ is conformally related to $(r^2g_+)|_{(T\partial X)^2}$. Thus, a CC manifold $(X^{n+1},g_+)$ induces a canonical conformal class $[g_+]_\infty:= [(r^2g_+)|_{(T\partial X)^2}]$ on $\partial X$. This is an invariant of $(X^{n+1},g_+)$, and the conformal manifold $(\partial X,[g_+]_\infty)$ is known as the conformal infinity. This is explained in \cite{Graham}, where some context and motivation from physics are given. Furthermore, it was shown in \cite{mazzeo1986hodge} that, as $r\to 0^+$,
\begin{equation}\label{RiemTensor-HolCoord}
    (R_{g_+})_{ijkl} = -|dr|^2_{\overline g}[(g_+)_{ik}(g_+)_{jl} - (g_+)_{il}(g_+)_{jk}] + O_{ijkl}(r^{-3}),
\end{equation}
where $O_{ijkl}(r^{-3})$ denotes the components of a four tensor. An important a consequence of (\ref{RiemTensor-HolCoord}) is that the sectional curvatures of a conformally compact manifold tend to $-|dr|^2_{\bar g}(y)$ as we approach any point $y\in \partial X$ at the boundary. Therefore, even though $|dr|^2_{\bar g}$ depends on the defining function $r$, its restriction $|dr|^2_{\bar g}|_{\partial X}$ to $\partial X$ is an invariant of $(X^{n+1},g_+)$; see Section 2 in \cite{Graham}. This motivates the next definition: 

\begin{definition}[Asymptotically Hyperbolic Manifold]
Asymptotically hyperbolic manifolds are those CC spaces for which $|dr|^2_{\bar g}|_{\partial X}$ is constant and equals to $1$ on $\partial X$, i.e, asymptotically hyperbolic manifolds are conformally compact manifolds for which the sectional curvatures approach $-1$ at $\partial X$. 
\end{definition}

In the case of hyperbolic space $\Hyp$, $r(x) =\frac{1-|x|}{1+|x|}$ works as a defining function, and $(r^2g_H)|_{(T\mathbb{S}^n)^2}$ is a constant multiple of the standard round $g_{\mathbb{S}^n}$ metric on $\mathbb{S}^n$. Therefore, $(\mathbb{S}^n,g_{\mathbb{S}^n})$ is the conformal infinity of $\Hyp$. Moreover, it is asymptotically hyperbolic as its sectional curvatures are $-1$ everywhere. The given defining function satisfies the special property that $|dr|^2_{r^2g_H} \equiv 1$ not only on $\mathbb{S}^n$ but also on a neighborhood of $\mathbb{S}^n$ in $\Hyp$ (everywhere except at the origin in this case). This allows us to write the hyperbolic metric $g_H$ as $g_H = r^{-2}(dr^2+g_r)$, where $g_r$ is a one parameter family of metrics on $\mathbb{S}^n$ given by $g_r = 4^{-1}(1-r^2)^2g_{\mathbb{S}^n}$; see Section 2 in \cite{Graham}.

More generally, it is a result of Graham and Lee \cite{GrahamLee} that given any $\hat g\in [g_+]_{\infty}$ in the conformal infinity of an AH manifold, there is a defining function $r$ such that $(r^2g_+)|_{(T\partial X)^2} = \hat g$ and $|dr|^2_{\bar g}\equiv 1$ on a neighborhood of $\partial X$; a defining function with such a property is called a special or geodesic defining function.  This induces an identification of this neighborhood of $\partial X$ in $X^{n+1}$ with $\partial X\times [0,\epsilon_r)$, $\epsilon_r>0$, on which the compactified metric can be written as 
\begin{equation}\label{NormalForm}
\bar g = r^2g_+ = dr^2+g_r
\end{equation}
for a one-parameter family of metrics on $\partial X$ with $g_0 = \hat g$; see Section 2 in \cite{Perez-AyalaSamuel2025HaCc}. The metric as written in (\ref{NormalForm}) is said to be in normal form.  This coordinate system at infinity (see Figure \ref{fig:holographic-coords}) is sometimes referred to as holographic coordinates \cite{Perez-AyalaSamuel2025HaCc}, and it is relevant in the derivation of eigenvalue upper bounds. 

\begin{figure}[H]
    \centering
\begin{tikzpicture}[>=Stealth, thick, scale=1.3]

    \coordinate (Center) at (-5, 0); 
    
    \fill[red!10] (0,-1.5) arc (-18:18:4.85) -- (-1.2,1.15) arc (16:-16:4.15) -- cycle;

    \draw (0,-1.5) arc (-18:18:4.85) coordinate[pos=0.6] (P) node[at end, above right] {$\partial X$};
    
    \draw[blue, dashed] (-1.2,-1.15) arc (-16:16:4.15);
    \node at (-1.8, 0) {$X^{n+1}$};
    
    \coordinate (VEnd) at ($(P)!1.2cm!(Center)$);
    \draw[blue, ->] (P) -- (VEnd);
    
    \node[black, anchor=south, inner sep=4pt] at ($(P)!0.6cm!(Center)$) {$\nabla_{\bar{g}} r$};
    
    \begin{scope}[shift={(P)}]
        \pgfmathanglebetweenpoints{\pgfpointanchor{P}{center}}{\pgfpointanchor{Center}{center}}
        \let\angle\pgfmathresult
        \begin{scope}[rotate=\angle]
            \draw[blue, thin] (0.2,0) -- (0.2,0.2) -- (0,0.2);
        \end{scope}
    \end{scope}

    \draw[->] (1.3, 0) -- (2.8, 0);

    \begin{scope}[shift={(5.5,0)}]
        \fill[red!10] (-1.5, -1.3) rectangle (1.5, 0.4);

        \draw[<->] (-1.5, -1.3) -- (1.5, -1.3) node[right] {$\partial X$};
        
        \draw[blue, ->] (0, -1.3) -- (0, 1.5) node[right, black] {$r$};
        
        \draw[blue, dashed, <->] (-1.5, 0.4) -- (1.5, 0.4);
        
        \node[blue, fill=red!10, inner sep=1.5pt, font=\small] at (0, 0.4) {$\epsilon_r$};
        
        \node at (0, 1.8) {$\partial X \times [0, \epsilon_r)$};
    \end{scope}

\end{tikzpicture}
\caption{Coordinates near the boundary of an AH manifold.}
\label{fig:holographic-coords}
\end{figure}
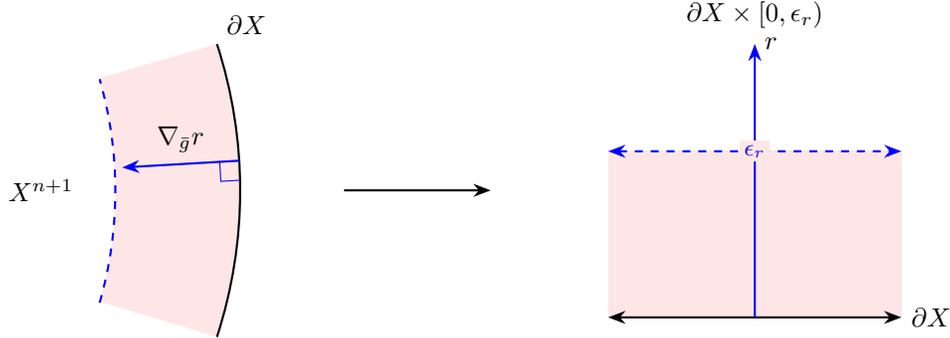

We conclude this subsection with a brief discussion of other conformally compact spaces. On any smooth compact Riemannian manifold $(\overline{X}^{n+1}, \bar g)$, there is always a unique defining function $r$ for the boundary $\partial X$ such that $g_+:=r^{-2}\bar g$ satisfies $R_g = -n(n+1)$ on $\text{int}(\overline X)$, however the regularity is delicate; see the introduction of \cite{GrahamC.Robin2017Vrfs}. These are examples of AH metrics. The AdS-Schwarzchild spaces are given by $(\mathbb{R}^2\times \mathbb{S}^2,g^m)$, where 
\[
g^m = V\,dt^2 + V^{-1}\,dr^2 + r^2 g_{\mathbb{S}^2}, \quad 
V = 1 + r^2 - \frac{2m}{r},
\]
with $m>0$ and $r\in [r_+,\infty)$, and where $r_+$ is the first positive root of $V(r) = 0$; see details in \cite{ChangSun-YungA.2007SPiC}. These are examples of Poincar\'e-Einstein (PE) spaces, i.e. CC manifolds with $\Ric_{g^m}\ge -3g^m$, whose sectional curvatures are not everywhere bounded above by $-1$; indeed, direct computation shows that, near $r_+$, some sectional curvatures are positive. A natural generalization of AdS-Schwarzchild spaces to higher dimensions and with nonspherical factors is explained in \cite{Anderson:2004yi}. Finally, quotients of hyperbolic space $(\mathbb{B}^{n+1}/\Gamma,g)$ which are, for instance, geometrically finite without cusps, are examples of nonsimply connected, PE spaces; the spectral properties of these are studied in \cite{SullivanDennis1987Raop}.

\subsection{Statement of the results}\label{DefinitionsStatements} 
Our main new contributions are (i) sharp upper bounds for the first Dirichlet eigenvalue of the $p$-Laplacian operator on conformally compact spaces; (ii)  a spectral rigidity result for asymptotically CMC submanifolds inside AH spaces; (iii) sharp upper bounds for the first eigenvalue of the polyharmonic operators under both clamped and buckling boundary conditions; and (iv) sharp lower bounds for a large class of weakly Poincar\'e-Einstein manifolds and their submanifolds. Regarding (iii), we remark that, to our knowledge, no Cheng-type inequality has been established for general polyharmonic operators under either clamped or buckling boundary conditions.

We now proceed to discuss these results, beginning with the upper bounds.

\subsubsection{Sharp eigenvalue upper bounds and a rigidity result}
In previous work by the author, it was shown that on an asymptotically hyperbolic manifold it holds that $\lambda_{1,p}(X^{n+1},g_+)\le \left(\frac{n}{p}\right)^p = \lambda_{1,p}(\Hyp,g_H)$ for any $1<p<\infty$; see Theorem 1.2 in \cite{PerezAyalaTyrrell}. Our first result is a generalization of this upper bound to a sharp upper bound on any conformally compact manifold, that is, we generalize it to the case where the asymptotic sectional curvatures are allowed to vary.

\begin{theorem}\label{UpperBound-CC}
    Let $(X^{n+1},g_+)$ be a conformally compact manifold, and let $r$ be a defining function for $\partial X$. Then 
    \begin{equation}\label{1stEigenEstimate}
        \lambda_{1,p}(X^{n+1},g_+) \le \left(\min|dr|_{\bar g}|_{\partial X}\right)^p\left(\frac{n}{p}\right)^{p} = \lambda_{1,p}(\mathbb{H}^{n+1}(-\min|dr|^2_{\bar g}|_{\partial X})),
    \end{equation}
    where $1<p<\infty$. Here $\mathbb{H}^{n+1}(-\min|dr|^2_{\bar g}|_{\partial X})$ denotes the simply connected manifold with constant sectional curvatures equal to $-\min|dr|^2_{\bar g}|_{\partial X}$. 
\end{theorem}

Note that the equality in (\ref{1stEigenEstimate}) is a simple rescaling argument; see Corollary 1.3 in \cite{PerezAyalaTyrrell}. The estimate (\ref{1stEigenEstimate}) was known in the case of $p=2$ owing to the work of Mazzeo in \cite{MazzeoRafe1991UCaI}, see also Theorem 1.1 in the work of Charalambous--Rowlett \cite{CharalambousNelia2024TLso} and references therein. To our knowledge, the result is new in all other cases ($p\not = 2$). The key ingredient in the proof of Theorem \ref{UpperBound-CC} is a Cheng-type inequality for the first $p$-Dirichlet eigenvalue proved by Takeuchi \cite{Takeuchi} together with asymptotic estimates for the $p$-Dirichlet eigenvalue on geodesic balls in hyperbolic space given in \cite{jin2024lowerbounddirichleteigenvalue}.

A simple consequence of Theorem \ref{UpperBound-CC} is that if $\lambda_{1,p}(X^{n+1},g_+) = \lambda_{1,p}(\Hyp)$, then the asymptotic sectional curvatures are bounded above by $-1$. This gives a rigidity result on a large class of submanifolds of asymptotically hyperbolic spaces that were introduced in our previous work \cite{PerezAyalaTyrrell} (Definition 1.6). We call these submanifolds asymptotically CMC  submanifolds - immersed submanifolds $\iota: Y^{k+1}\to (X^{n+1},g_+)$ which are conformally compact themselves, meet $\partial X$ transversely along  $\partial Y\subset \partial X$. Furthermore, the (unnormalized) mean curvature $H^Y$ satisfies $g_+(H^Y,H^Y) = C^2+O(r)$ for a defining function $r$. Finally, we assume that these submanifolds are $C^{m,\alpha}$ for $m\geq 3, \alpha\in (0,1).$

\begin{corollary}\label{UpperBound-CC-Sub}
    Let $(X^{n+1},g_+)$ be an asymptotically hyperbolic manifold and let $Y^{k+1}$ be an asymptotically CMC submanifold. If $\lambda_{1,p}(Y^{k+1}) = \lambda_{1,p}(\mathbb{H}^{k+1}(-1))$, then $(Y^{k+1},h_+)$ is asymptotically hyperbolic. Equivalently, $(Y^{k+1},h_+)$ is asymptotically minimal and meets $\partial X$ orthogonally. 
\end{corollary}
\begin{remark}
In \cite{PerezAyalaTyrrell}, Corollary 1.11, it was shown that if $Y^{k+1}$ is a minimal conformally compact submanifold of $\Hyp$, then $\lambda_{1,p}(Y^{k+1}) = \lambda_{1,p}(\mathbb{H}^{k+1}(-1))$. Corollary \ref{UpperBound-CC-Sub} above provides a partial converse inside more general AH spaces, where ``minimal'' is replaced by the much weaker condition ``asymptotically minimal''.
\end{remark}
We would like to emphasize that the submanifold $Y^{k+1}$ in Corollary \ref{UpperBound-CC-Sub} is only assumed to be conformally compact, therefore its asymptotic sectional curvatures could vary. Thus, knowing that the bottom of the spectrum is that of hyperbolic space forces the asymptotic sectional curvature to be constant and equal to $-1$, it forces its mean curvature to vanish at $\partial Y$, and it forces $Y^{k+1}$ to meet $\partial X$ orthogonally. Using Kac's famous analogy \cite{KacMark1966COHt} ``Can one hear the shape of the drum?'', the asymptotic sectional curvatures, the asymptotic mean curvature and the meeting angle at infinity of an asymptotically CMC submanifold inside an asymptotically hyperbolic manifold can all be heard. 

We now proceed with a discussion on upper bounds for $\Gamma^l_1(X^{n+1},g_+)$ and $\Lambda_1(X^{n+1},g_+)$. To our knowledge, there is no Cheng-type inequality analogue for $\lambda_1^C((-\Delta_{g_+}),\Omega)$ or $\lambda_1^B((-\Delta_{g_+}),\Omega)$. However, our techniques require only the manifold to be asymptotically hyperbolic. The second main result of this work is as follows:

\begin{theorem}\label{Sharp-UpperBounds}
    Let $(X^{n+1},g_+)$ be an AH manifold. Then
    \begin{equation}\label{Sharp-UB-Clamped}
        \Gamma^l_1(X^{n+1},g_+)\le \left(\frac{n}{2}\right)^{2l} = \Gamma^l_1(\mathbb{H}^{n+1}(-1)),
    \end{equation}
    and 
    \begin{equation}\label{Sharp-UB-Buckling}
        \Lambda_1(X^{n+1},g_+)\le \frac{n^2}{4} = \Lambda_1(\mathbb{H}^{n+1}(-1)).
    \end{equation}
\end{theorem}

The equality in both estimates demonstrates the sharpness of our results; it follows from these upper bound estimates combined with the lower bound estimates presented in the next subsection. Both inequalities (\ref{Sharp-UB-Clamped}) and (\ref{Sharp-UB-Buckling}) appear to be new in the literature. We note that AH manifolds could have positive curvature somewhere, and also they might not be simply connected, as noted in the introduction. However, the upper bounds are determined by the geometry near infinity, which is that of the model case $\Hyp$. 

We observe  that there is no assumption on the connectedness of $\partial X$. \cref{Sharp-UpperBounds} holds as long as $|dr|^2_{\bar g}|_{(T\partial X)^2}\equiv 1$ on one boundary component, even if it varies on other boundary components. The key idea is to apply Graham-Lee's result locally, and write $\bar g$ as in (\ref{NormalForm}) near that boundary component. Thus, the AH assumption in Theorem \ref{Sharp-UpperBounds} can be relaxed to only CC and require that the asymptotic sectional curvatures on one boundary component are constant and equal to $-1$. Finally, it is clear from the proof of Theorem \ref{Sharp-UpperBounds} that the same techniques can be used to obtain sharp upper bounds for more general clamped and buckling-type eigenvalue problems such as the ones considered in \cite{BuosoLamberti2013}.

The last part of our results concerns lower bounds for all three types of eigenvalue problems. 

\subsubsection{Sharp eigenvalue lower bounds}

Before discussing our results on the lower bounds, we make some further remarks about conformally compact manifolds. 

In addition to asymptotically hyperbolic manifolds, another important class of conformally compact manifolds is that of Poincar\'e-Einstein (PE) manifolds\footnote{These are also called Conformally Compact Einstein manifolds in the literature.}. PE manifolds are CC manifolds that also satisfy $\Ric_{g_+} = -ng_+$. It follows by contracting in (\ref{RiemTensor-HolCoord}) that any PE manifold is AH. When writing a PE metric in normal form, as in (\ref{NormalForm}), the expansion of $g_r$ is even, up to $n-1$ if $n$ is odd and up to $n$ if $n$ is even (Lemma 4 in \cite{ChangSun-YungA.2007SPiC}), and the first terms in the expansion are
\begin{equation}\label{WPE}
    \bar g = dr^2+g_r = dr^2+ \hat g - P_{\hat g}\cdot r^2 + O(r^3)
\end{equation}
as $r\to 0^+$, where $P_{\hat g}$ is the Schouten tensor of the induced boundary metric $\hat g$; see Section 2 in \cite{Graham}. Asymptotically hyperbolic manifolds whose metric has the same expansion as in (\ref{WPE}) in normal form are called weakly Poincar\'e-Einstein (WPE) - this is because the expansion matches that of a PE metric up to, and including, order $2$. In summary, 
\[
\{\text{PE}\}\subset \{\text{WPE}\}\subset \{\text{AH}\}\subset \{\text{CC}\}.
\]

Our results concern WPE manifolds satisfying $\Ric_{g_+}\ge -ng_+$, and whose conformal infinity has a nonnegative Yamabe constant $Y(\partial X, [g_+]_\infty)$. 

\begin{theorem}\label{Sharp-LowerBounds-pLap}
   Let $(X^{n+1},g_+)$ be a weakly Poincar\'e-Einstein manifold with $\Ric_{g_+}\ge -ng_+$, and assume that $Y(\partial X, [g_+]_\infty)\ge 0$.  
   \begin{enumerate}
   \item Then 
   \begin{equation}\label{Sharp-LB-pLap-Amb}
   \lambda_{1,p}(X^{n+1},g_+) = \left(\frac{n}{p}\right)^p.
   \end{equation}
   \item Let $\iota: Y^{k+1}\to X^{n+1}$ be a complete and noncompact immersion into $(X^{n+1},g_+)$ whose mean curvature satisfies $|H^Y|\le \alpha$ for some constant $\alpha$. If the submanifold invariant $\hat \beta^Y$ (see Definition \ref{Submanifold-Invariant}) satisfies $\alpha+\hat \beta^Y< k$, then  
   \begin{equation}\label{Sharp-LB-pLap-Sub}
       \lambda_{1,p}(Y^{k+1})\ge \left(\frac{k-\alpha - \hat \beta^Y}{p}\right)^p
   \end{equation}
   for any $1<p<\infty$. Here, $\hat \beta^Y$ is an invariant of $(Y^{k+1},\iota^*g_+)$ which is always nonnegative for noncompact submanifolds. 
   \end{enumerate}
\end{theorem}

Let us begin with a discussion of (\ref{Sharp-LB-pLap-Amb}). Due to Theorem \ref{UpperBound-CC}, this is essentially a lower bound result. Theorem \ref{Sharp-LowerBounds-pLap}, part (1), is an extension to any $p>1$ of a result by Lee \cite{LeeJohnM.1995Tsoa} in the case of the Laplacian ($p=2$). Lee proved that for PE manifolds with $Y(\partial X,[g_+]_\infty)\ge 0$, we have $\lambda_{1,2}(X^{n+1},g_+) = \left(\frac{n}{2}\right)^2$, although it became clear later that his results also hold under our assumption; see the discussion in \cite{GuillarmouColin2010SCoP}. Theorem \ref{Sharp-LowerBounds-pLap} first appeared in \cite{HijaziOussama2020TCCo}, where it was proven with a different technique; see Theorem 9 there. The fact that their assumptions are equivalent to ours is the content of Theorem 1.1 in \cite{Perez-AyalaSamuel2025HaCc}. Finally, the key idea for our argument is the availability in these spaces of a special ``cut-off'' function, which we have already utilized in \cite{Perez-AyalaSamuel2025HaCc}, together with key observations also made in \cites{DuMao,LinHezi2025Sgef} and later, and more generally, in \cite{FarkasKajantoKristaly2025}. Theorem \ref{Sharp-LowerBounds-Clamped&Buckling} below also uses this special ``cut-off'' function, showing the universality of our methods.

We will now discuss the submanifold case, i.e. estimate (\ref{Sharp-LB-pLap-Sub}). In the work by Du--Mao \cite{DuMao}, it was proven that for the same class of submanifolds, if the ambient space $(X^{n+1},g_+) = \Hyp$ is hyperbolic space, then $\lambda_{1,p}(Y^{k+1})\ge \left(\frac{k-\alpha}{p}\right)^p$. We understand (\ref{Sharp-LB-pLap-Sub}) as a generalization of Du--Mao's result, where the error term $\hat\beta^Y$ measures, in a way, the deviation from the ambient manifold from being $\Hyp$. In fact, the invariant $\hat \beta^Y$ was introduced in \cite{PerezAyalaTyrrell}, where it was proven that $\hat \beta^Y = 0$ if the ambient space is hyperbolic space; see the proof of Corollary 1.15 there. In an upcoming work, the authors and collaborators use this invariant to study, among other things, the nonexistence of closed minimal submanifolds within a large class of CC spaces.

We obtain a similar result for the other two eigenvalue problems.

\begin{theorem}\label{Sharp-LowerBounds-Clamped&Buckling}
   Let $(X^{n+1},g_+)$ be a weakly Poincar\'e-Einstein manifold with $\Ric_{g_+}\ge -ng_+$, and assume that $Y(\partial X, [g_+]_\infty)\ge 0$.
   \begin{enumerate}
   \item Then 
   \begin{equation}
   \Gamma^l_1(X^{n+1},g_+) = \left(\frac{n}{2}\right)^{2l}\;\;\;\;\text{and}\;\;\;\; \Lambda_1(X^{n+1},g_+)=\frac{n^2}{4}.
   \end{equation}
   \item Let $\iota: Y^{k+1}\to X^{n+1}$ be a complete and noncompact immersion into $(X^{n+1},g_+)$ whose mean curvature satisfies $|H^Y|\le \alpha$ for some constant $\alpha$. If the submanifold invariant $\hat \beta^Y$ satisfies $\alpha+\hat \beta^Y< k$, then  
   \begin{equation}\label{Sharp-LB-Clamped-Sub}
       \Gamma^l_1(Y^{k+1})\ge \left(\frac{k-\alpha - \hat \beta^Y}{2}\right)^{2l}\;\;\;\; \text{and}\;\;\;\; \Lambda_1(Y^{k+1})\ge \left(\frac{k-\alpha - \hat\beta^Y}{2}\right)^2.
   \end{equation}
   \end{enumerate}
\end{theorem}
These results are new for this class of spaces. Thanks to Theorem \ref{Sharp-UpperBounds}, this result is essentially about the lower bounds
\begin{equation}\label{LowerBounds-CB}
    \Gamma_1^l(X^{n+1},g_+)\ge \left(\frac{n}{2}\right)^{2l} \;\;\;\; \text{and}\;\;\;\; \Gamma_1(X^{n+1},g_+)\ge \frac{n^2}{4}
\end{equation}
We would like to mention some prior results in this direction. On simply connected manifolds with sectional curvatures bounded above, everywhere, by $-1$, estimate (\ref{LowerBounds-CB}) for the clamped eigenvalue with respect to the biharmonic operator ($l=2$) was shown in \cite{KristalyAlexandru2020Ftoc} under the assumption that the so called ``$\kappa$-Cartan--Hadamard conjecture'' holds; see Theorem 1.1 and Section 2.2 there. Later in  \cite{FarkasKajantoKristaly2025}, both estimates in (\ref{LowerBounds-CB}) were shown under the assumption that the manifold is simply connected with sectional bounded curvatures bounded above by $-1$; see Theorem 4.1 and Theorem 1.2 there. Our result is on manifolds that are not necessarily simply connected and whose sectional curvatures might change sign somewhere.

The results for submanifolds are the same in spirit as those for submanifolds for the $p$-Laplace operator. In the work of Lin \cite{LinHezi2025Sgef}, the author obtained similar results when the ambient manifold is simply connected, noncompact and with sectional curvatures bounded above everywhere by $-\kappa^2$ ($\kappa>0$), but with no $\hat \beta^Y$ appearing; see Theorems 4.2 and Theorem 4.3 there. We remark once again that our ambient manifold is not simply connected and that its sectional curvatures might change sign somewhere, even though at the boundary they approach $-1$ uniformly. Notably, it is worth noting that their condition on the sectional curvatures implies the opposite bound on the Ricci curvature.

\begin{remark}
    As mentioned, it was proven in \cite{PerezAyalaTyrrell} that $\hat \beta^Y = 0$ for any noncompact immersed submanifold of $\Hyp$. It would be interesting to see if this result can be generalized to any simply connected manifold with sectional curvatures bounded above by $-1$ everywhere.
\end{remark}

\section{Preliminaries}

Our techniques rely on a delicate analysis of the asymptotic expansions of several geometric quantities on an asymptotically hyperbolic manifold. In this section, we will compute some of the necessary expansions. 

Let us suppose that $(X^{n+1},g_+)$ is an AH manifold. As explained in the introduction, by a result of Graham-Lee \cite{GrahamLee}, there exists a special defining function $r$ for $\partial X$ such that $|dr|^2_{r^2 g_+}\equiv 1$ on a neighborhood of $\partial X$ in $X^{n+1}$. This neighborhood can be identified with $\partial X\times [0,\E_r)$, $\E_r>0$, and the metric can be written in normal form as in (\ref{NormalForm}). On $\partial X\times [0,\E_r)$, we thus have
\[
   dv_{g_+}=r^{-n-1} \bigg(\frac{\mathrm{det}g_r}{\mathrm{det} g_0}\bigg)^{1\slash  2} dv_{\hat g} dr,
\]
where 
\[
\bigg(\frac{\mathrm{det}g_r}{\mathrm{det}g_0}\bigg)^{1\slash  2}=1+(-H^{\partial X})r+O(r^2).
\]
Here, $H^{\partial X}$ is the mean curvature of $\partial X$ with respect to $\overline{g}$; see Section 2 in \cite{PerezAyalaTyrrell}. Importantly, $dv_{g_+}= r^{-n-1}(1+O(r))dv_{\hat g}dr$ as $r\to0^+$.

Our cutoff function involves the special defining function $r$, and so the following expansions will be needed. We start with the expansion of the differential operator $\Delta_{g_+}$.

\begin{lemma}\label{ExpansionLaplacian}
Suppose $(X^{n+1}, g_+)$ is an AH manifold and $r$ is a special defining function for it. In normal form $\bar g = r^2g_+ = dr^2 + g_r$ in a collar neighborhood of $\partial X$, we write
\begin{equation}
    g_r = \hat g + g_{(1)}\cdot r + g_{(2)}\cdot r^2 + \cdots
\end{equation}
for the expansion of $g_r$. Then 
\begin{equation}
    \Delta_{g_+} = -(n-1)r\cdot \partial_r + r^2\cdot\left(\Delta_{\hat g} + \partial^2_{rr} + \frac{1}{2}\tr_{\hat g}(g_{(1)})\partial_r\right) +  O(r^3)E,
\end{equation}
where $E$ is a differential operator that is first order in the radial direction and second order in the tangential directions.
\end{lemma}

\begin{remark}
    The coefficient tensors $g_{(1)}$, $g_{(2)}$, and $g_{(3)}$ have been computed by many researchers, see \cite{Perez-AyalaSamuel2025HaCc} and the references therein. However, an explicit expression is not needed here. 
\end{remark}

\begin{proof}
    We have 
\[
    \begin{split}
    \Delta_{g_+} &= g_+^{\A\B}(\partial^2_{\A\B} - \Gamma_{\A\B}^r(g_+)\partial_r - \Gamma_{\A\B}^\G(g_+)\partial_\G) + r^2(\partial^2_{rr} - \Gamma_{rr}^r(g_+)\partial_r - \Gamma_{rr}^\A(g_+)\partial_\A) \\ &= r^2(\hat g^{\A\B} - g_{(1)}^{\A\B}\cdot r + O(r^2))\times\\&\hspace{1in}\left(\partial^2_{\A\B} - \left(\hat g_{\A\B}r^{-1}+\frac12 (g_{(1)})_{\A\B} + O_{\A\B}(r^2)\right)\partial_r - \left(\Gamma_{\A\B}^{\G}(\hat g) + O_{\A\B}^{\G}(r)\right)\partial_\G\right) \\&\;\;\;\;+r^2(\partial^2_{rr} - (-r^{-1})\partial_r) \\ &= -(n-1)r\cdot \partial_r+ r^2\cdot\left(\Delta_{\hat g} +\partial^2_{rr} - \frac12\hat g^{\A\B}(g_{(1)})_{\A\B}\partial_r +g^{\A\B}_{(1)}\hat g_{\A\B}\partial_r\right)  +  O(r^3)E,
    \end{split}
\]  
from which the result follows. 
\end{proof}

We proceed with the computation of the polyharmonic operator $(-\Delta_{g_+})^l$ acting on $r$.

\begin{lemma}\label{Delta^m-on-r}
Let $r$ be a special defining function. Then
\begin{equation}
    \Delta^l_{g_+}r =(-(n-1))^l\cdot r +O(r^2),
\end{equation}
for $l\in \mathbb{N}$ as $r\to0^+$.
\end{lemma}

\begin{proof}
    Let us compute as follows. From Lemma \ref{ExpansionLaplacian}, we have
    \[
    \Delta_{g_+}r =-(n-1)\cdot r +\frac12\tr_{\hat g}(g_{(1)})\cdot r^2 + O(r^3) = -(n-1)\cdot r + O(r^2).
    \]
    Assume that it holds for $l=k\in\mathbb{N}$. Applying the Laplacian operator to $\Delta_{g_+}^kr$, we obtain
    \[
    \begin{split}
        \Delta^{k+1}_{g_+}r &= -(n-1)r^k\cdot \left\{-(n-1)\cdot r+ O(r^2)\right\} +(r^2),
    \end{split}
    \]
    and so the result follows by induction. 
\end{proof}

We observe that, in the case of the polyharmonic operator $(-\Delta_{g_+})^l$ with $l$ odd, the numerator in the Rayleigh quotient (\ref{Rayleigh-Quotient-B}) is of the form $|\nabla_{g_+}\Delta^m_{g_+}r|^2$ for $m\in\mathbb{N}$. This motivates our next lemma.
\begin{lemma}\label{Nabla-Delta^m-on-r}
    Let $r$ be a special defining function. Then 
    \begin{equation}
        \nabla_{g_+}\Delta^m_{g_+}r = (-(n-1))^mr^2\partial_r + O(r^3)E,
    \end{equation}
    where $E$ is a linear differential operator. Moreover,
    \begin{equation}
        |\nabla_{g_+}\Delta_{g_+}^mr|^2 = (n-1)^{2m}r^2 + O(r^3).
    \end{equation}
\end{lemma}

\begin{proof}
    From Lemma \ref{Delta^m-on-r}, we have
    \[
    \begin{split}
        \nabla_{g_+}\Delta^m_{g_+}r &= \nabla_{g_+}\left((-(n-1))^m\cdot r +O(r^2)\right) \\ &= g_+^{\A\B}\partial_\A((-(n-1))^m\cdot r +O(r^2))\partial_\B+r^2\partial_r((-(n-1))^m\cdot r +O(r^2))\partial_r \\&= O^\B(r^3)\partial_\B + (-(n-1))^m r^2\partial_r + O(r^3)\partial_r,
    \end{split}
    \]
    as desired. Furthermore, 
    \[
    \begin{split}
        |\nabla_{g_+}\Delta^m_{g_+}r|^2 &= g_+^{\A\B}\partial_\A(\Delta^m_{g_+}r)\partial_\B(\Delta^m_{g_+}r) + r^2(\partial_r(\Delta^m_{g_+}r))^2  \\ &= O(r^6) + r^2((-(n-1))^{m}+ O(r))^2,
    \end{split}
    \]
    from which the result follows. 
\end{proof}

\subsection{Family of test functions}

In this section, we introduce the family of test-functions that we will use to derive upper bounds for the clamped and buckling eigenvalues. Additionally, we will derive the needed asymptotic estimates. 

Let $r$ be a special defining function, and choose $\epsilon_r>0$ such that $|dr|^2_{r^2g_+}\equiv 1$ on $\{p\in X: r(p)\le \epsilon_r\}$ - see the discussion at the beginning of the current section. This is equivalent to having $|\nabla_{g_+}r|^2\equiv r^2$ on $\{p\in X: r(p)\le \epsilon_r\}$. For $\epsilon\in (0,\epsilon_r)$ and $\delta>0$, consider the cutoff function 
    \begin{equation}\label{CutOff-Function}
\phi_{\epsilon\delta}(p):=\begin{cases} 
      1, & \quad\epsilon\leq r(p) \le \E_r \\
      \frac{1+\delta}{\epsilon}(r(p)-\frac{\epsilon\delta}{1+\delta}), & \frac{\epsilon\delta}{1+\delta}\leq r(p)\leq \epsilon \\
      0, & \quad 0\leq r(p)\leq   \frac{\epsilon\delta}{1+\delta}.
   \end{cases}
\end{equation}
We also require $\phi_{\E\D}(p) \equiv 0$ for $r(p)\ge A$, where $A>\epsilon_r$. How we choose $\phi_{\E\D}$ to decay to zero in this region does not affect the computations. This cutoff function was introduced in \cite{PerezAyalaTyrrell} to derive upper bounds for the $p$-Laplacian operator on asymptotically hyperbolic manifolds. The family of test functions we consider is $r^{s}\phi_{\E\D}$, where $s\in (\frac{n-1}{2},\frac{n}{2})$. 

Our goal is to compute both $R^C_l(r^s\phi_{\E\D})$ and $R^B(r^s\phi_{\E\D})$. This is the motivation for the lemmas within this section. In what follows, we say that a function $w$ is $O(a,b) = O(a(\epsilon), b(\delta))$ if there exist positive constants $A$, $B$, $\epsilon_o$ and $\delta_o$ such that $|w(\epsilon,\delta)|\le A|a(\epsilon)|$ and $|w(\epsilon,\delta)|\le B|b(\delta)|$ for all $0\le \epsilon<\epsilon_o$ and for all $0\le \delta<\delta_o$. $O(a,b,c)$ is defined similarly.

\begin{lemma}\label{CutOff-Estimates}
    Let the setup be as in Lemma \ref{ExpansionLaplacian} and let $m\in \mathbb{N}\cup \{0\}$. For any $s>0$, we have
    \begin{equation}\label{CutOff-Estimates1}
        \Delta^m_{g_+}(r^s\phi_{\E\D}) = 
    \begin{cases}
    \frac{1+\D}{\E}(s+1)(s+n-1)r^{s+1} +O(r^s,\D) & \;\text{ if }\;m=1\\
    \frac{1+\D}{\E}(s+1)^m(s+1-n)^mr^{s+1} + O(r^{s+2},\E^{-1}) + O(r^s,\D) &\;\text{ if }\;m>1,
    \end{cases}
    \end{equation}
    while
    \begin{equation}\label{CutOff-Estimates2-1}
         |\nabla_{g_+}\Delta_{g_+}(r^s\phi_{\E\D})|^2 = \left(\frac{1+\D}{\E}\right)^2(s+1)^4(s+1-n)^2r^{2s+2} +O(r^{2s+1},\E^{-1},\D)+O(r^{2s},\D^2),
    \end{equation}
    and, if $m>1$,
    \begin{equation}\label{CutOff-Estimates2}
    \begin{split}
    |\nabla_{g_+}\Delta_{g_+}^m(r^s\phi_{\E\D})|^2 &= \left(\frac{1+\D}{\E}\right)^2(s+1)^{2m+2}(s+1-n)^{2m}r^{2s+2}+O(r^{2s+3},\E^{-2})\\&\;\;\;\;+ O(r^{2s+1},\E^{-1},\D) + O(r^{2s+2},\E^{-1},\D)+O(r^{3s+2},\E^{-2},\D) +O(r^{2s},\D),
    \end{split}
    \end{equation}
    on $\left\{p\in X: \frac{\E\D}{1+\D}<r(p)<\E\right\}$. Additionally, (\ref{CutOff-Estimates2}) holds for $m=0$ but only with the $O(r^{2s},\D)$ term. On the other hand, 
    \begin{equation}\label{CutOff-Estimates3}
       \Delta_{g_+}^m(r^s) = s^m(s-n)^mr^s + O(r^{s+1})
    \end{equation}
    and
    \begin{equation}\label{CutOff-Estimates4}
       |\nabla_{g_+}\Delta^m_{g_+}(r^s)|^2 = s^{2m+2}(s-n)^{2m}r^{2s} + O(r^{2s+1}) 
    \end{equation}
    as $r\to 0^+$. Furthermore, (\ref{CutOff-Estimates4}) holds for $m=0$ with no Big-O term.
\end{lemma}

\begin{proof}
    We start by arguing (\ref{CutOff-Estimates3}) first. On a neighborhood of $\partial X$, we have
    \[
    \Delta_{g_+}(r^s) = \Div(sr^{s-1}\nabla_{g_+}r) = sr^{s-1}(-(n-1)r+O(r^2)) + s(s-1)r^{s-2}\cdot r^2,
    \]
    where we have used that $|\nabla_{g_+}r|^2\equiv r^2$ on a neighborhood of $\partial X$. Therefore, 
    \begin{equation}\label{CutOff-Estimates5}
        \Delta_{g_+}(r^s) = s(s-n)r^s + O(r^{s+1})
    \end{equation}
    as $r\to 0^+$. Assume that (\ref{CutOff-Estimates3}) holds for $m=k$. Computing the Laplacian of $\Delta^k(r^s)$ yields 
    \[
    \begin{split}
    \Delta_{g_+}^{k+1}(r^s) &= \Delta_{g_+}(s^k(s-n)^kr^s + O(r^{s+1}))\\ &= s^k(s-n)^k\cdot\left\{s(s-n)r^s + O(r^{s+1})\right\},
    \end{split}
    \]
    from which the claim follows. This shows
    (\ref{CutOff-Estimates3}). For (\ref{CutOff-Estimates4}), we have
    \[
    \begin{split}
    |\nabla_{g_+}\Delta_{g_+}^m(r^s)|^2 &= r^2(\partial_r(\Delta^m_{g_+}(r^s)))^2 \\&= r^2(s^{m+1}(s-n)^mr^{s-1} + O(r^{s}))^2 \\& = s^{2m+2}(s-n)^{2m}r^{2s} + O(r^{2s+1}),
    \end{split}
    \]
    as $r\to0^+$, which gives the result. Observe that, if $m=0$, then 
    \[
    |\nabla_{g_+}(r^s)|^2 = s^2r^{2s-2}|\nabla_{g_+}r|_{g_+} = s^2r^{2s},
    \]
    as desired.

    We proceed now with the proof of (\ref{CutOff-Estimates1}) and (\ref{CutOff-Estimates2}). On $\frac{\epsilon \D}{1+\D}< r(p)< \epsilon$, we have
    \begin{equation}
        \nabla_{g_+}\phi_{\E\D} = \frac{1+\D}{\E}\cdot \nabla_{g_+}r
    \end{equation}
    and 
    \begin{equation}
        \Delta_{g_+}\phi_{\E\D} = -(n-1)\frac{1+\D}{\E}\cdot r + O(r^2)
    \end{equation}

    We now proceed as follows. First, 
    \[
    \begin{split}
        \Delta_{g_+}(r^s\phi_{\E\D}) &= r^{s}\Delta_{g_+}\phi_{\E\D} + \phi_{\E\D}\Delta_{g_+}(r^s) + 2sr^{s-1}g_+(\nabla_{g_+}r, \nabla_{g_+}\phi_{\E\D}) \\ &= -(n-1)\frac{1+\D}{\E}r^{s+1} + s(s-n)\frac{1+\delta}{\epsilon}\left(r-\frac{\epsilon\delta}{1+\delta}\right)r^s + 2sr^{s+1}\frac{1+\D}{\E} + O(r^{s+1},\D) \\&= s(n-s)\D r^s + \frac{1+\D}{\E}r^{s+1}(-(n-1)+s(s-n)+2s) + O(r^{s+1}, \D).
    \end{split}
    \]
    That is,
    \begin{equation}
        \Delta_{g_+}(r^s\phi_{\E\D}) = O(r^s,\D) + \frac{1+\D}{\E}(s+1)(s+1-n)r^{s+1}.
    \end{equation}
    Applying the Laplacian again and using (\ref{CutOff-Estimates5}) gives
    \[
    \begin{split}
        \Delta_{g_+}^2(r^s\phi_{\E\D}) &= \frac{1+\D}{\E}(s+1)(s+1-n)\Delta_{g_+}(r^{s+1}) + O(r^s,\D) \\& = \frac{1+\D}{\E}(s+1)^2(s+1-n)^2r^{s+1}+ O(r^{s+2},\E^{-1}) + O(r^s,\D).
    \end{split}
    \]
    Once again,
    \[
    \Delta^3_{g_+}(r^s\phi_{\E\D}) = \frac{1+\D}{\E}(s+1)^3(s+1-n)^3r^{s+1} + O(r^{s+2},\E^{-1}) + O(r^s,\D).
    \]
    Therefore, it follows inductively it holds that
    \[
    \Delta^m_{g_+}(r^s\phi_{\E\D}) = 
    \begin{cases}
    \frac{1+\D}{\E}(s+1)(s+n-1)r^{s+1} +O(r^s,\D) & \;\text{ if }\;m=1,\\
    \frac{1+\D}{\E}(s+1)^m(s+1-n)^mr^{s+1} + O(r^{s+2},\E^{-1}) + O(r^s,\D) &\;\text{ if }\;m>1,
    \end{cases} 
    \]
    as claimed.

    Finally, as $\Delta^m_{g_+}(r^s\phi_{\E\D})$ is a radial function, we compute the norm squared of its gradient as follows. For $m=1$, we have
    \[
    \begin{split}
        |\nabla_{g_+}\Delta_{g_+}(r^s\phi_{\E\D})|^2 &= r^2(\partial_r\Delta_{g_+}(r^s\phi_{\E\D}))^2 \\&=r^2\left(\frac{1+\D}{\E}(s+1)^2(s+n-1)r^{s}+ O(r^{s-1},\D)\right)^2 \\&=\left(\frac{1+\D}{\E}\right)^2(s+1)^4(s+1-n)^2r^{2s+2} +O(r^{2s+1},\E^{-1},\D)+O(r^{2s},\D^2).
    \end{split}
    \]
    On the other hand, for $m>1$, 
    \[
    \begin{split}
    |\nabla_{g_+}\Delta_{g_+}^m(r^s\phi_{\E\D})|^2 &= r^2(\partial_r\Delta_{g_+}^m(r^s\phi_{\E\D}))^2 \\ & = r^2\left(\frac{1+\D}{\E}(s+1)^{m+1}(s+1-n)^mr^s+O(r^{s+1},\E^{-1}) + O(r^{s-1},\D)\right)^2 \\&= \left(\frac{1+\D}{\E}\right)^2(s+1)^{2m+2}(s+1-n)^{2m}r^{2s+2}+O(r^{2s+3},\E^{-2})+ O(r^{2s+1},\E^{-1},\D) \\&\;\;\;\;+ O(r^{2s+2},\E^{-1},\D)+O(r^{3s+2},\E^{-2},\D) +O(r^{2s},\D),
    \end{split}
    \]
    as desired. Note that if $m=0$, then 
    \[
    \begin{split}
    |\nabla_{g_+}(r^s\phi_{\E\D})|^2 &= |r^s\nabla_{g_+}\phi_{\E\D} + sr^{s-1}\phi_{\E\D}\nabla_{g_+}r|^2 \\& = r^{2s+2}\left(\frac{1+\D}{\E}\right)^2 + s^2r^{2s}\left(\frac{1+\D}{\E}\right)^2\left(r - \frac{\E\D}{1+\D}\right)^2 \\&\;\;\;\;+ 2sr^{2s+1}\left(\frac{1+\D}{\E}\right)^2\left(r - \frac{\E\D}{1+\D}\right) \\&= \left(\frac{1+\D}{\E}\right)^2(s+1)^2r^{2s+2} + O(r^{2s},\D),
    \end{split}
    \]
    as claimed.
\end{proof}

\section{Proof of the Main Results}

\subsection{Proof of Theorem \ref{UpperBound-CC} and Corollary \ref{UpperBound-CC-Sub}}

\begin{proof}[Proof of Theorem \ref{UpperBound-CC}]
    Let $y_m\in M$ be a point where $|dr|_{\bar g}|_{\partial X}$ attains its minimum value $\alpha_m$. For any $\epsilon>0$ and  $R>0$, there exists a point $p$ near $y_m$ such that $\Ric_{g_+}\ge -\alpha_m^2(1+\epsilon)ng_+$ on the geodesic ball $B(p,R)$. By Takeuchi \cite{Takeuchi} (see Theorem 2), 
    \begin{equation}
        \lambda_{1,p}(X^{n+1},g_+)\le \lambda_{1,p}(B(p,R),g_+) \le \lambda_{1,p}(V(-\alpha_m^2(1+\epsilon),R),g_H),
    \end{equation}
    where $V(-\alpha_m^2(1+\epsilon),R)$ is a geodesic ball of radius $R$ in $\mathbb{H}^{n+1}(-\alpha_m^2(1+\epsilon))$. In particular, 
    \begin{equation}
        \lambda_{1,p}(X^{n+1},g_+)\le \lim_{R\to\infty} \lambda_{1,p}(V(-\alpha_m^2(1+\epsilon),R)).
    \end{equation}
    From Lemma 4.2 in \cite{jin2024lowerbounddirichleteigenvalue}, we deduce that
    \begin{equation}
        \lim_{R\to\infty} \lambda_{1,p}(V(-\alpha_m^2(1+\epsilon),R),g_H) = \alpha_m^p(1+\epsilon)^{\frac{p}{2}}\left(\frac{n}{p}\right)^p.
    \end{equation}
    The result now follows after taking the limit as $\epsilon$ goes to zero. 
\end{proof}

\begin{proof}[Proof of Corollary \ref{UpperBound-CC-Sub}]
    As a first step, we claim that the asymptotic sectional curvatures of $Y^{k+1}$ are bounded below by $-1$. Indeed, let $\tilde r:= r|_{\overline Y}$ be the restriction to $\overline Y^{k+1}$ of a defining function $r$ for $(X^{n+1},g_+)$. Then $|d\tilde r|_{\bar h}|_{\partial Y}\le |dr|_{\bar g}|_{\partial Y} = 1$, and thus the claim follows. On the other hand, under the assumption that $\lambda_{1,p}(Y^{k+1}) = \left(\frac{k}{p}\right)^p$, it follows from Theorem \ref{UpperBound-CC}, applied to the CC manifold $(Y^{k+1},h_+)$, that $(\min |d\tilde r|_{\bar h}|_{TY})^p\ge 1$. Therefore, $|d\tilde r|_{\hat h}|_{\partial Y}\equiv 1$, thus $(Y^{k+1},h_+)$ is AH.

    Finally, the mean curvature of $(Y^{k+1},h_+)$ vanishes at $\partial Y$ (i.e. $C=0$), following Proposition 1.8 in \cite{PerezAyalaTyrrell}. Similarly, the orthogonality at infinity follows from the proof of the same Proposition 1.8; see section 4.1 in \cite{Perez-AyalaSamuel2025HaCc}.
\end{proof}

\subsection{Proof of Theorem \ref{Sharp-UpperBounds}}

\begin{proof}[Proof of Theorem \ref{Sharp-UpperBounds}, Clamped boundary conditions]
According to (\ref{Poly-Eigenvalue-Def}) and (\ref{VariationalCharacterization}), we have
    \[
    \Gamma^l_1(X^{n+1}, g_+)\le R^C_l(r^{s}\phi_{\E\D})
    \]
    for all $\epsilon\in(0,\epsilon_r)$ and $\D>0$, where $\phi_{\E\D}$ is the cutoff function introduced in (\ref{CutOff-Function}) and $s$ is such that $2s-n\not = 0$. Note that $r^{s}\phi_{\E\D}\in W^{2,l}_0(\Omega)$, where $\Omega$ is a bounded domain containing the support of $\phi_{\E\D}$. Assume $-1<2s - n< 0$, that is, $\frac{n-1}{2}<s<\frac{n}{2}$. In the end, we will let $\D\to 0^+$ first, and then $\E\to0^+$, so for expressions like $O(\E^p,\D^q)$ we need to keep track of only the powers of $\D$. 
    
    We start with the expansion of the denominator in $R^C_l(r^s\phi_{\E\D})$, the integral of $(r^s\phi_{\E\D})^2$:
    \[
    \begin{split}
         &\;\;\;\;O(1) + \int_\E^{\E_r}\int_{\partial X} r^{2s-n-1}(1+O(r))\;dv_{\hat g}dr \\&\;\;\;\;+ \left(\frac{1+\D}{\E}\right)^2\int_{\frac{\E\D}{1+\D}}^\E\int_{\partial X}\left(r-\frac{\E\D}{1+\D}\right)^2r^{2s-n-1}(1+O(r))\;dv_{\hat g}dr \\&= O(1) + \Vol_{\hat g}(\partial X)\cdot \left(\left[\frac{r^{2s-n}}{2s-n}\right|^{\E_r}_{\E} + O(\E^{2s-n+1})\right) \\&\;\;\;\; + \left(\frac{1+\D}{\E}\right)^2\Vol_{\hat g}(\partial X)\int_{\frac{\E\D}{1+\D}}^\E\left(r^2-\frac{2r\E\D}{1+\D}+\left(\frac{\E\D}{1+\D}\right)^2\right)r^{2s-n-1}(1+O(r))\;dr \\&= O(1) - \Vol_{\hat g}(\partial X)\frac{\E^{2s-n}}{2s-n} + \Vol_{\hat g}(\partial X)\frac{(1+\D)^2}{\E^2(2s-n+2)} \left(\E^{2s-n+2} - \left(\frac{\E\D}{1+\D}\right)^{2s-n+2}\right)\\&\;\;\;\;+ \Vol_{\hat g}(\partial X)\left(\frac{1+\D}{\E}\right)^2\left(O(\E^{2s-n+2},\D^{2s-n+2}) + O(\E^{2s-n+2},\D^2) \right) \\&= O(1) - \Vol_{\hat g}(\partial X)\frac{\E^{2s-n}}{2s-n} + \frac{\Vol_{\hat g}(\partial X)}{2s-n+2}\E^{2s-n}+ O(\D^{2s-n+2}) + O(\D^2) \\ &= O(1) + O(\D^{2s-n+2}) + \Vol_{\hat g}(\partial X)\E^{2s-n}\left(\frac{1}{2s -n + 2} - \frac{1}{2s-n}\right)
    \end{split}
    \]
    After taking $\D\to 0^+$, we obtain
    \begin{equation}\label{Sharp-UB-Clamped1}
    \lim_{\D \to 0}\int_X |r^s\phi_{\E\D}|^2\;dv_{g_+} = O(1) + \Vol_{\hat g}(\partial X)\E^{2s-n}\left(\frac{1}{2s -n + 2} - \frac{1}{2s-n}\right)
    \end{equation}

    Let us focus on the numerator. According to (\ref{Rayleigh-Quotient-C}), we have two cases.

    \vspace{.15in}
    \underline{(Case 1: $l=2m$)} We first assume that $l=2m$, so that the numerator involves the integral of $|\Delta^m_{g_+}(r^s\phi_{\E\D})|^2$. According to Lemma \ref{CutOff-Estimates}, we need to treat the cases of $m=1$ and $m>1$ separately. Starting with $m=1$, estimates (\ref{CutOff-Estimates1}) and (\ref{CutOff-Estimates3}) gives that the integral of $|\Delta_{g_+}(r^s\phi_{\E\D})|$ equals
    \[
    \begin{split}
        &O(1)+\Vol_{\hat g}(\partial X)\int_\E^{\E_r}(s^{2}(s-n)^{2}r^{2s}+O(r^{2s+1}))r^{-n-1}(1+O(r))\;dr \\&\;\;\;\;+ \Vol_{\hat g}(\partial X)\int_{\frac{\E\D}{1+\D}}^\E \left(\left(\frac{1+\D}{\E}\right)^2(s+1)^2(s+n-1)^2r^{2s+2} + O(r^{2s},\D)\right)r^{-n-1}(1+O(r))\;dv_{g_+},
    \end{split}
    \]
    where we omit writing the term $O(r^{2s+1},\E^{-1},\D)$ because, as explained, we will let $\D\to0^+$ first. Continuing, we observe
    \[
    \begin{split}
        \int_X|\Delta_{g_+}(r^s\phi_{\E\D})|^2\;dv_{g_+} &= O(1) - \Vol_{\hat g}(\partial X)s^2(s-n)^2\frac{\E^{2s-n}}{2s-n} + O(\E^{2s-n+1}) \\&\;\;\;\;+\Vol_{\hat g}(\partial X)(s+1)^2(s+n-1)^2\frac{\E^{2s-n}}{2s-n+2} + O(\D) + O(\D^{2s-n+2}).
    \end{split}
    \]
    In the case of $m>1$, the integral of $|\Delta^m_{g_+}(r^s\phi_{\E\D})|^2$ gives
    \[
    \begin{split}
 &O(1) + \Vol_{\hat g}(\partial X)\int_\E^{\E_r}(s^{2m}(s-n)^{2m}r^{2s}+O(r^{2s+1}))r^{-n-1}(1+O(r))\;dr\\&\;\;\;\;+\Vol_{\hat g}(\partial X)\int_{\frac{\E\D}{1+\D}}^\E \Bigg(\frac{1+\D}{\E}(s+1)^m(s+1-n)^mr^{s+1} \\&\hspace{2in}+ O(r^{s+2},\E^{-1}) + O(r^s,\D)\Bigg)^2 r^{-n-1}(1+O(r))\;dr \\ &= O(1) - \Vol_{\hat g}(\partial X)s^{2m}(s-n)^{2m} \frac{\E^{2s-n}}{2s-n} + O(\E^{2s-n+1})\\&\;\;\;\;+\Vol_{\hat g}(\partial X)(s+1)^{2m}(s+1-n)^{2m}\frac{\E^{2s-n}}{2s-n+2} +O(\D)+O(\D^{2s-n+2}).
    \end{split}
    \]
    Therefore, in either case, $m=1$ or $m>1$, we always obtain
    \begin{equation}\label{Sharp-UB-Clamped2}
    \begin{split}
    \lim_{\D\to0^+}\int_X&|\Delta^m_{g_+}(r^s\phi_{\E\D})|^2\;dv_{g_+} = O(1) + O(\E^{2s-n+1}) \\&\hspace{1.235in}+\Vol_{\hat g}(\partial X)\E^{2s-n}\left(\frac{(s+1)^{2m}(s+1-n)^{2m}}{2s-n+1} - \frac{s^{2m}(s-n)^{2m}}{2s-n}\right)
    \end{split}
    \end{equation}
    Recall that $-1<2s-n<0$. Putting (\ref{Sharp-UB-Clamped1}) and (\ref{Sharp-UB-Clamped2}), we deduce
    \[
    \begin{split}
    \Gamma^l_1(X^{n+1},g_+)&\le \lim_{\E\to0^+}\lim_{\D\to0+}R^C_l(r^s\phi_{\E\D}) \\&\le \frac{\displaystyle\left(\frac{(s+1)^{2m}(s+1-n)^{2m}}{2s-n+1} - \frac{s^{2m}(s-n)^{2m}}{2s-n}\right)}{\displaystyle\left(\frac{1}{2s -n + 2} - \frac{1}{2s-n}\right)}\\&= \frac{1}{2}\left(s^{2m}(s-n)^{2m}(2s-n+2) - \frac{(s+1)^{2m}(s+1-n)^{2m}(2s-n+2)(2s-n)}{2s-n+1}\right)
    \end{split}
    \]
    for all $s\in (\frac{n-1}{2},\frac{n}{2})$. After taking $s\to \frac{n}{2}^-$, we obtain $\Gamma_l(X^{n+1},g_+)\le \left(\frac{n}{2}\right)^{4m} = \left(\frac{n}{2}\right)^{2l}$.
    
\vspace{.15in}
    \underline{(Case 2: $l=2m+1$)} In this case, and according to (\ref{Rayleigh-Quotient-C}), the numerator is the integral of $|\nabla_{g_+}\Delta^m_{g_+}(r^s\phi_{\E\D})|^2$. As before, we start with the case when $m=1$. Recall that for terms of the form $O(\E^p,\D^q)$, we only need to keep track of the powers of $\D$ as we are taking $\D\to0^+$ first. Using estimates (\ref{CutOff-Estimates2-1}) and (\ref{CutOff-Estimates4}) in Lemma \ref{CutOff-Estimates}, we obtain that the integral of $|\nabla_{g_+}\Delta_{g_+}(r^s\phi_{\E\D})|^2$ equals 
    \[
    \begin{split}
        &O(1) + \Vol_{\hat g}(\partial X)\int_{\E}^{\E_r}\left(s^{4}(s-n)^{2}r^{2s} + O(r^{2s+1}) \right)r^{-n-1}(1+O(r))\;dr \\ &+ \Vol_{\hat g}(\partial X)\int_{\frac{\E\D}{1+\D}}^{\E}\Bigg(\left(\frac{1+\D}{\E}\right)^2(s+1)^4(s+1-n)^2r^{2s+2} +O(r^{2s+1},\E^{-1},\D)\\&\hspace{2.5in}+O(r^{2s},\D^2)\Bigg)r^{-n-1}(1+O(r))dr \\&= O(1) - s^4(s-n)^2\frac{\E^{2s-n}}{2s-n} + O(\E^{2s-n+1}) \\&\;\;\;\;+ \Vol_{\hat g}(\partial X)(s+1)^4(s+1-n)^2\frac{\E^{2s-n}}{2s-n+2} + O(\D^{2s-n+2}) + O(\D).
    \end{split}
    \]
    In summary,
    \[
    \begin{split}
    \int_X|\nabla_{g_+}\Delta_{g_+}(r^s\phi_{\E\D})|^2\;dv_{g_+} &= O(1) - s^4(s-n)^2\frac{\E^{2s-n}}{2s-n} + O(\E^{2s-n+1}) \\&\;\;\;\;+ \Vol_{\hat g}(\partial X)(s+1)^4(s+1-n)^2\frac{\E^{2s-n}}{2s-n+2} + O(\D^{2s-n+2}) + O(\D).
    \end{split}
    \]
    For the case where $m>1$, we use estimates (\ref{CutOff-Estimates2}) and (\ref{CutOff-Estimates4}) in Lemma \ref{CutOff-Estimates} to compute the integral of $|\nabla_{g_+}\Delta^m_{g_+}(r^s\phi_{\E\D})|^2$. We obtain
    \[
    \begin{split}
        &O(1) - s^{2m+2}(s-n)^{2m}\frac{\E^{2s-n}}{2s-n}+ O(\E^{2s-n+1}) \\&+ \Vol_{\hat g}(\partial X) (s+1)^{2m+2}(s+1-n)^{2m}\frac{\E^{2s-n}}{2s-n+2} + O(\D^{2s-n+2}) + O(\D).
    \end{split}
    \]
    Once again, in either case, $m=1$ or $m>1$, we deduce
    \begin{equation}\label{Sharp-UB-Clamped3}
    \begin{split}
        \lim_{\D\to 0^+} \int_X|\nabla_{g_+}&\Delta^m_{g_+}(r^s\phi_{\E\D})|^2\;dv_{g_+} = O(1)+ O(\E^{2s-n+1})\\&\;\;\;\;+ \Vol_{\hat g}(\partial X)\E^{2s-n}\left(\frac{(s+1)^{2m+2}(s+1-n)^{2m}}{2s-n+2} -\frac{s^{2m+2}(s-n)^{2m}}{2s-n}\right)
    \end{split}
    \end{equation}
Recalling that $-1<2s-n<0$, and combining (\ref{Sharp-UB-Clamped1}) and (\ref{Sharp-UB-Clamped3}), we obtain that
\[
\begin{split}
    \Gamma^l_1(X^{n+1},g_+)&\le \lim_{\E\to0^+}\lim_{\D\to 0^+}R^C_l(r^s\phi_{\E\D}) \\& \le \frac{\displaystyle\left(\frac{(s+1)^{2m+2}(s+1-n)^{2m}}{2s-n+2} -\frac{s^{2m+2}(s-n)^{2m}}{2s-n}\right)}{\displaystyle\left(\frac{1}{2s -n + 2} - \frac{1}{2s-n}\right)} \\ &=\frac{1}{2}\left((2s-n+2)s^{2m+2}(s-n)^{2m} - (2s-n)(s+1)^{2m+2}(s+1-n)^{2m}\right)
\end{split}
\]
holds for all $s\in (\frac{n-1}{2}, \frac{n}{2})$. After taking $s\to\frac{n}{2}^-$, we obtain $\Gamma_l(X_{n+1},g_+)\le \left(\frac{n}{2}\right)^{4m+2} = \left(\frac{n}{2}\right)^{2l}$, as desired.
\end{proof}

\begin{proof}[Proof of Theorem \ref{Sharp-UpperBounds}, Buckling boundary conditions]
 We use the same test function $r^{s}\phi_{\E\D}$. The estimates that we need for the denominator in $R^B(r^s\phi_{\E\D})$ are in Lemma \ref{CutOff-Estimates}, specifically (\ref{CutOff-Estimates2}) and (\ref{CutOff-Estimates4}) with $m=0$. The computations that we need have already been performed; from (\ref{Sharp-UB-Clamped3}) with $m=0$, we obtain 
 \begin{equation}\label{Sharp-UB-Buckling1}
     \lim_{\D\to0^+}\int_X|\nabla_{g_+}(r^s\phi_{\E\D})|^2 = O(1)+O(\E^{2s-n+1}) + \Vol_{\hat g}(\partial X)\E^{2s-n}\left(\frac{(s+1)^2}{2s-n+2} - \frac{s^2}{2s-n}\right).
 \end{equation}
 For the numerator in $R^B(r^s\phi_{\E\D})$, we use estimates in Lemma (\ref{CutOff-Estimates}) with $m=1$. Once again, the estimate that we need has already been done; from (\ref{Sharp-UB-Clamped2}), we obtain
 \begin{equation}\label{Sharp-UB-Buckling2}
 \begin{split}
     \lim_{\D\to0^+}\int_X |\Delta_{g_+}(r^s\phi_{\E\D})|^2\;dv_{g_+} &=  O(1) + O(\E^{2s-n+1}) \\&\;\;\;\;+ \Vol_{\hat g}(\partial X)\left(\frac{(s+1)2(s+1-n)^2}{2s-n+1} - \frac{s^2(s-n)^2}{2s-n}\right)
 \end{split}
 \end{equation}
Putting (\ref{Sharp-UB-Buckling1}) and (\ref{Sharp-UB-Buckling2}) yields
\[
\begin{split}
\Lambda_1(X^{n+1},g_+) &\le \lim_{\E\to 0^+}\lim_{\D\to0^+}R^B(r^s\phi_{\E\D}) \\&\le \frac{\displaystyle\left(\frac{(s+1)^2(s+1-n)^2}{2s-n+1} - \frac{s^2(s-n)^2}{2s-n}\right)}{\displaystyle\left(\frac{(s+1)^2}{2s-n+2} - \frac{s^2}{2s-n}\right)} \\&= 
\frac{(2s - n + 2)\left(s^2 (n - s)^2 (2s - n + 1) + (n - 2s)(s+1)^2 (s+1 - n)^2\right)}{(2s - n + 1)\left(s^2 (2s - n + 2) + (n - 2s)(s+1)^2\right)}
\end{split}
\]
for all $s\in(\frac{n-1}{2}, \frac{n}{2})$. The result follows after taking $s\to \frac{n}{2}^-$.
\end{proof}

\subsection{Proof of Theorems \ref{Sharp-LowerBounds-pLap} and \ref{Sharp-LowerBounds-Clamped&Buckling}}

On any asymptotically hyperbolic manifold, given a smooth defining function $r$, there is a unique, smooth, and strictly positive function on $X^{n+1}$ satisfying
\begin{equation}\label{LeeEigenfunction}
\begin{cases}
    \Delta_{g_+}u = (n+1)u &\text{on } X^{n+1},\\  u = r^{-1}+O(1) & \text{as } r\to0^+, 
\end{cases}
\end{equation}
see Proposition 4.1 in \cite{LeeJohnM.1995Tsoa}\footnote{The apparent discrepancy with Lee's work is due to the differences in the sign convention in the definition of the Laplace operator. Recall that for us $\Delta_{g_+}$ is defined as a negative operator.}. We call a solution of (\ref{LeeEigenfunction}) a Lee-eigenfunction. On a weakly Poincar\'e-Einstein manifold with $\Ric_{g_+}\ge -ng_+$, the next term in the expansion of $u$ is given by 
\begin{equation}\label{LeeEigen-Expansion}
u = r^{-1} + \frac{\hat R}{4n(n-1)}\cdot r^2 + O(r^3).
\end{equation}
From (\ref{LeeEigen-Expansion}) it then follows that 
\begin{equation}\label{LeeEigen-Expansion2}
    u^2 - |\nabla_{g_+}u|^2 = \frac{\hat R}{n(n-1)} + o(1).
\end{equation}
See Lemma 2.1 and the proof of Lemma 2.2 in \cite{GuillarmouColin2010SCoP}, together with remarks afterwards, where both (\ref{LeeEigen-Expansion}) and (\ref{LeeEigen-Expansion2}) are derived. 

Using these expansions, a maximum principle argument gives that if the Yamabe constant of the conformal infinity $Y(\partial X,[g_+]_\infty)$ is nonnegative, then we can take a defining function $r$ for which $\hat g = (r^2g_+)|_{(T\partial X)^2}$ has nonnegative scalar curvature, and so the associated Lee-eigenfunction satisfies the gradient estimates 
\begin{equation}\label{GradientEstimates}
    |\nabla_{g_+}u|^2<u^2
\end{equation}
everywhere on $X^{n+1}$; see Theorem A in \cite{LeeJohnM.1995Tsoa}, or the refinement in \cite{GuillarmouColin2010SCoP}. We use $u$ as a cutoff function to derive eigenvalue lower bounds.

\begin{remark}
    Equation (\ref{LeeEigenfunction}) is a specific case of the more general family of Poisson equations
    \[
    -\Delta_{g_+}u = s(n-s)u.
    \]
    If $s(n-s)$ is not in the point spectrum of $-\Delta_{g_+}$ and $\text{Re}(s)>\frac{n}{2}$, $s\not\in \frac{n}{2}+\mathbb{N}$, then given $f\in C^{\infty}(\partial X)$, there exist a unique solution $u_s$ satisfying $u_s = Fr^{n-s}+Gr^s$. Here, $F,G\in C^{\infty}(\overline X)$ and $F|_{\partial X} = f$. The map sending each $f\in C^{\infty}(\partial X)$ to its unique solution $u_s$ is called the Poisson operator, while the map sending $f\in C^{\infty}(\partial X)$ to $G|_{\partial X}$ is known as the scattering operator; see \cites{Graham-Zworski, CaseJeffreyS.2016OFGO}. Lee studies the case where $s=n+1$.
\end{remark}

Let $Y^{k+1}$ be a complete immersed submanifold of $X^{n+1}$, $u$ a Lee-eigenfunction and denote by $h_+$ the induced metric on $Y^{k+1}$. Set $\tilde u = u|_Y$, that is, $\tilde u$ is the restriction of $u$ to $Y$. Using Lemma 2 in \cite{ChoeJaigyoung1992Iiom}, we have
\[
\begin{split}
\Delta_{h_+}\tilde u &= (\Delta_{g_+}u)|_Y + H^Y(u) - \tr_{g_+}(\nabla^2_{g_+}u)|_{(TY^\perp)^2} \\ &= (n+1)\tilde u + H^Y(u) - \tr_{g_+}((\mathring{\nabla}^2_{g_+}u)|_{(TY^\perp)^2}) - (n-k)\tilde u \\&= (k+1)\tilde u + H^Y(u) - \tr_{g_+}((\mathring{\nabla}^2_{g_+}u)|_{(TY^\perp)^2}),
\end{split}
\]
where in the second to last line we have used $\mathring{\nabla}^2_{g_+}u = \nabla^2_{g_+}u - ug_+$. For eigenvalue estimates on submanifolds, it is crucial to understand the term involving the normal trace of the trace-free hessian of $u$. This motivates the following definition:
\begin{definition}\label{Submanifold-Invariant}
    For a complete immersed submanifold $\iota: Y^{k+1}\to X^{n+1}$ inside a CC manifold $(X^{n+1},g_+)$ for which a Lee-eigenfunction $u$ exists, we define \[\beta^Y(u) := \sup_Y(u^{-1}\tr_{g_+}((\mathring{\nabla}^2_{g_+}u)|_{(TY^\perp)^2})).\] Additionally, set $\hat \beta^Y:= \inf \beta^Y(u)$, where the infimum is being taken over all possible Lee-eigenfunctions. We note that $\hat \beta^Y$ is an invariant of $(Y^{k+1},\iota^*g_+)$.
\end{definition}

We now have all the ingredients for the proofs. It is clear that the proof of Theorem \ref{Sharp-LowerBounds-pLap} reduces to the following Poincar\'e-type inequalities. 

\begin{proposition}\label{PoincareTypeInequality}
    Let $(X^{n+1},g_+)$ be a weakly Poincar\'e-Einstein manifold with $\Ric_{g_+}\ge -ng_+$, and assume that $Y(\partial X, [g_+]_\infty)\ge 0$. Then,
    \begin{enumerate}
        \item For any smooth domain $\Omega\subset M$ and $f\in C^\infty_0(\Omega)\setminus\{0\}$ nonnegative, we have 
        \begin{equation}\label{PI-Manifold}
        \int_\Omega |\nabla_{g_+}f|^p \ge \left(\frac{n}{p}\right)^p\int_\Omega f^p.
        \end{equation}
        \item If $(Y^{k+1},h_+)\subset (X^{n+1},g_+)$ is a complete and noncompact submanifold, for any $\Omega\subset Y^{k+1}$ and $f\in C^\infty_0(\Omega)\setminus\{0\}$ nonnegative, we have 
        \begin{equation}\label{PI-SubManifold}
        \int_\Omega |\nabla_{h_+}f|^p \ge \left(\frac{k- \beta^Y(u) - \alpha}{p}\right)^p \int_\Omega f^p.
        \end{equation}
        As in Theorem \ref{Sharp-LowerBounds-pLap}, we assume that $\alpha + \beta^Y(u)<k$.
    \end{enumerate}
\end{proposition}

\begin{proof}
Throughout this proof, $u$ denotes a Lee-eigenfunction satisfying the gradient estimates (\ref{GradientEstimates}). For any smooth nonnegative function, we compute as follows:
    \[
    \begin{split}
        \Div(f^p \nabla_{g_+} (\ln u)) &= g_+(\nabla_{g_+}(f^p), u^{-1}\nabla_{g_+}u) + f^{p}\Delta_{g_+}(\ln u) \\&= p f^{p - 1}g_+(\nabla_{g_+}f, u^{-1}\nabla_{g_+}u) + f^p((n+1) - u^{-2}|\nabla_{g_+}u|^2)\\ &\ge -pf^{p -1}|\nabla_{g_+}f| + nf^p \\&\ge -\frac{p\epsilon^qf^{q(p - 1)}}{q} - \frac{p|\nabla_{g_+}f|^p}{p\epsilon^p} + nf^p,
    \end{split}
    \]
    where $q$ is the H\"older conjugate of $p$, i.e. $\frac{1}{p}+\frac{1}{q} = 1$, and where $\epsilon>0$. Let $\Omega$ be any smooth bounded domain in $X^{n+1}$. Applying the divergence theorem and using $|\partial_n(\ln u)|\le 1$ yields
    \[
    \begin{split}
        \int_{\partial \Omega} f^p \;dv_{\sigma_+}&\ge -\frac{p\epsilon^q}{q}\int_\Omega f^{p}\;dv_{g_+} - \frac{1}{\epsilon^p}\int_\Omega |\nabla_{g_+}f|^p\;dv_{g_+} + n\int_\Omega f^p\;dv_{g_+} \\ &= (1-p)\epsilon^{\frac{p}{p-1}}\int_\Omega f^p\;dv_{g_+} -\frac{1}{\epsilon^p}\int_\Omega |\nabla_{g_+}f|^p\;dv_{g_+} + n\int_\Omega f^p\;dv_{g_+},
    \end{split}
    \]
    that is,
    \[
        \frac{1}{\epsilon^p}\int_\Omega |\nabla_{g_+}f|^p\;dv_{g_+} + \int_{\partial \Omega} f^p\;dv_{\sigma_+} \ge \left(n + (1-p)\epsilon^{\frac{p}{p-1}}\right)\int_\Omega f^p\;dv_{g_+}
    \]
    If we further assume that $f\in C^{\infty}_0(\Omega)$, then
    \[
    \begin{split}
      &\int_\Omega |\nabla_{g_+}f|^p\;dv_{g_+} \ge \left(n + (1-p)\epsilon^{\frac{p}{p-1}}\right)\epsilon^p\int_\Omega f^p\;dv_{g_+}\\ \iff& \frac{\displaystyle\int_\Omega |\nabla_{g_+}f|^p\;dv_{g_+}}{\displaystyle\int_\Omega f^p\;dv_{g_+}} \ge \left(n + (1-p)\epsilon^{\frac{p}{p-1}}\right)\epsilon^p =: w(\epsilon).
    \end{split}
    \]
    The function $w$ attains a positive maximum at $\epsilon_m = \left(\frac{n}{p}\right)^{1- \frac{1}{p}}$ with value of $\left(\frac{n}{p}\right)^p$. This shows (1).

    As for (2), we proceed as follow. Denote by $\tilde u = u|_Y$ the restriction of the Lee-eigenfunction $u$ to $Y^{k+1}$, thus it follows that $|\nabla_{h_+}\tilde u| \le \tilde u$ still holds on $Y^{k+1}$. Moreover, following \cite{Perez-AyalaSamuel2025HaCc} (see proof of Theorem 1.6), we have
    \begin{equation}\label{InducedLaplacian}
        \Delta_{h_+}(\ln \tilde u) \ge k - \beta^Y(u) - \alpha>0,
    \end{equation}
    where $\alpha>0$ bounds the norm of the mean curvature of $Y$. Therefore,
    \begin{equation}
        \Div_{h_+}(f^p \nabla_{h_+}(\ln \tilde u)) \ge -\frac{p\epsilon^qf^{q(p - 1)}}{q} - \frac{|\nabla_{h_+}f|^p}{\epsilon^p} + f^p(k - \beta^Y(u) - \alpha)
    \end{equation}
    Using a similar technique, if $f\in C^\infty_0(\Omega)\setminus\{0\}$ is nonnegative, then
    \[
        \frac{\displaystyle\int_\Omega |\nabla_{h_+}f|^p\;dv_{g_+}}{\displaystyle\int_\Omega f^p\;dv_{g_+}} \ge \left(k-\beta^Y(u)-\alpha + (1-p)\epsilon^{\frac{p}{p-1}}\right)\epsilon^p =: w(\epsilon),
    \]
    where $w$ attains a maximum at $\epsilon_m = \left(\frac{k-\beta^Y(u) - \alpha}{p}\right)^{\frac{p-1}{p}}$ with value $\left(\frac{k-\beta^Y(u) - \alpha}{p}\right)^p$. This step finishes the proof.
\end{proof}

The techniques here closely followed those in Cheng--Leung's work \cite{CheungLeung-Fu2001Eefs}, and the later work by Du--Mao \cite{DuMao}. The key observation is that we have a smooth function $\rho$ satisfying $|\nabla_{g_+}\rho|\le 1$ and $\Delta_{g_+}\rho\ge b$, for some constant $b\in \mathbb{R}$,  everywhere on $X^{n+1}$. In our case, this function is $\rho = \ln(u)$, where $u$ is a Lee-eigenfunction satisfying the gradient estimates (\ref{GradientEstimates}), while we use $\tilde \rho = \ln(u|_Y)$ when dealing with the submanifold case. 

Having a smooth function $\rho$ satisfying $|\nabla_{g_+}\rho|\le 1$ and $\Delta_{g_+}\rho\ge b$ allows you to obtain much more general functional inequalities by generalizing what we have observed while proving Proposition \ref{PoincareTypeInequality}. 
In fact, in \cite{LinHezi2025Sgef}, the following inequality has been proven under the assumption of having a smooth function $\rho$ satisfying $|\nabla_{g_+}\rho|\le 1$ and $\Delta_{g_+}\rho\ge b$: if $\Omega$ is a bounded domain, then for any $p>1$ and $f\in C^\infty_0(\Omega)$, one has 
    \[
        \int_{\Omega}|\Delta_{g_+}f|^p\;dv_{g_+} \ge \left(\frac{(p-1)b^2}{p^2}\right)^p\int_\Omega|f|^p\;dv_{g_+};
    \]
and
\[
\int_\Omega |\Delta_{g_+}f|^2\;dv_{g_+}\ge \frac{b^2}{4} \int_\Omega |\nabla_{g_+}f|^2\;dv_{g_+},
\]
see Theorem 3.1 and Theorem 3.2 in \cite{LinHezi2025Sgef}, and also the work of Farkas--Kaj\'anto--Krist\'aly in \cite{FarkasKajantoKristaly2025}. Iterating this inequality yields
\begin{equation}\label{Lin1}
    \int_{\Omega}|\Delta_{g_+}^mf|^p\;dv_{g_+}\ge \left(\frac{(p-1)b^2}{p^2}\right)^{pm}\int_{\Omega}|f|^p\;dv_{g_+}
\end{equation}
for any $m\in \mathbb{N}$. In contrast, for $m\in \mathbb{N}$, applying (\ref{PI-Manifold}) yields
\[
\int_\Omega |\nabla_{g_+}\Delta_{g_+}^mf|^p\;dv_{g_+}\ge \int_\Omega |\nabla_{g_+}|\Delta^m_{g_+}f||^p\;dv_{g_+}\ge \left(\frac{n}{p}\right)^p\int_\Omega |\Delta_{g_+}^mf|^p\;dv_{g_+},
\]
thus
\begin{equation}\label{Lin2}
    \int_\Omega |\nabla_{g_+}\Delta^m_{g_+}f|^p\;dv_{g_+}\ge\left(\frac{n}{p}\right)^p\left(\frac{(p-1)b^2}{p^2}\right)^{pm}\int_\Omega |f|^p\;dv_{g_+}.
\end{equation}
In the submanifold setting, using (\ref{PI-SubManifold}) instead of (\ref{PI-Manifold}), we obtain
\begin{equation}\label{Lin3}
    \int_\Omega |\nabla_{h_+}\Delta^m_{h_+}f|^p\;dv_{h_+}\ge\left(\frac{k - \beta^Y(u) - \alpha}{p}\right)^p\left(\frac{(p-1)b^2}{p^2}\right)^{pm}\int_\Omega |f|^p\;dv_{h_+},
\end{equation}
where $h_+$ is the induced metric on $Y^{k+1}$. 

We are now in position to prove Theorem \ref{Sharp-LowerBounds-Clamped&Buckling}.

\begin{proof}[Proof of Theorem \ref{Sharp-LowerBounds-Clamped&Buckling}, part (1)]
    Recall that, on weakly Poincar\'e-Einstein manifolds with $\Ric_{g_+}\ge -ng_+$, we have a Lee-eigenfunction $u$ satisfying the gradient estimates (\ref{GradientEstimates}); see discussion at the beginning of the current section. Therefore, if we set $\rho = \ln u$, then $|\nabla_{g_+}\rho|\le 1$ and $\Delta_{g_+}\rho\ge n$ everywhere on $X^{n+1}$. This allows us to use (\ref{Lin1}) and (\ref{Lin2})

    We have two cases on the basis of the parity of $l$. If $l=2m$, then applying (\ref{Lin1}) with $p=2$ and $b=n$ yields
    \[
    \int_{\Omega}|\Delta_{g_+}^mf|^2\;dv_{g_+}\ge \left(\frac{n^2}{4}\right)^{2m}\int_{\Omega}|f|^2\;dv_{g_+} = \left(\frac{n}{2}\right)^{2l}\int_\Omega |f|^2\;dv_{g_+},
    \]
    which implies the lower bound. On the other hand, if $l=2m+1$, then applying (\ref{Lin2}) with $p=2$ and $b=n$ gives
    \[
    \int_\Omega |\nabla_{g_+}\Delta^m_{g_+}f|^2\;dv_{g_+} \ge \left(\frac{n}{2}\right)^2\left(\frac{n^2}{4}\right)^{2m} \int_\Omega |f|^2\;dv_{g_+} = \left(\frac{n}{2}\right)^{2l}\int_\Omega |f|^2\;dv_{g_+},
    \]
    as wanted. This shows that $\Gamma^l_1(X^{n+1},g_+)\ge \left(\frac{n}{2}\right)^{2l}$. Since any weakly Poincar\'e-Einstein manifold is AH, by definition, the equality holds owing to Theorem \ref{Sharp-UpperBounds}, estimate (\ref{Sharp-UB-Clamped}).

    The lower bound on the buckling eigenvalue $\Gamma_1(X^{n+1},g_+)$ follows from Theorem 3.2 in \cite{LinHezi2025Sgef}, where we use $\ln u$ as $\rho$. Once again, the equality follows them from Theorem \ref{Sharp-UpperBounds}, estimate (\ref{Sharp-UB-Buckling}).
\end{proof}

\begin{proof}[Proof of Theorem \ref{Sharp-LowerBounds-Clamped&Buckling}, part (2)] Set $\tilde u = u|_{Y}$ and $\tilde \rho = \ln \tilde u$. Since $|\nabla_{h_+}\tilde u|\le \tilde u$ still holds on $Y^{k+1}$, we have $|\nabla_{h_+}\tilde \rho|\le 1$. On the other hand, from (\ref{InducedLaplacian}) we know that $\Delta_{h_+}\tilde \rho \ge b$ with $b=k-\beta^Y(u)-\alpha$. For the clamped eigenvalue problem, the result follows from (\ref{Lin1}), if $l$ is even, and from (\ref{Lin3}) if $l$ is odd. For the buckling eigenvalue, it follows once more from Theorem 3.2 in \cite{LinHezi2025Sgef}.
\end{proof}

\begin{remark}
    In the work of Lin \cite{LinHezi2025Sgef}, the function $\rho$ is taken to be the distance function on the manifold to a fixed point. Because of that, to get the appropriate estimates on $|\nabla \rho|$ and $\Delta \rho$, they assume the manifolds to be simply connected and to have sectional curvatures bounded above, everywhere, by $-1$.
\end{remark}

\bibliographystyle{amsplain}
\bibliography{bibliography.bib}

\end{document}